\documentclass
{article}
\usepackage{amssymb,amsthm,amsmath,latexsym}
\usepackage{systeme}
\DeclareMathOperator{\nullspace}{null}
\usepackage{float}
\usepackage{url}
\usepackage{tikz}
\usetikzlibrary{calc,3d,patterns}

\usepackage{quoting}
\newtheorem{theorem}{Theorem}
\newtheorem{lemma}{Lemma}
\newtheorem{corollary}{Corollary}
\newtheorem{prop}{Proposition}

\makeatletter
\renewcommand*\env@matrix[1][*\c@MaxMatrixCols c]{% 
  \hskip -\arraycolsep
  \let\@ifnextchar\new@ifnextchar
  \array{#1}}
\makeatother

\usepackage{hyperref}
\usepackage[justification=centering]{caption}

\begin{document}

\title
{A Gauche Perspective  on \\ Row Reduced Echelon Form
  \\     and Its Uniqueness
}
\date{\vspace{-5ex}}
\markright{ \, Gauche Perspective on  Echelon Form }

\author{Eric L. Grinberg }

%\newenvironment{acknowledgement}%       New acknowledgement environment
%    {\large\bfseries Acknowledgement%
 %   \par\medskip\normalfont\normalsize}%
%    {}%

\maketitle

%\author{ Eric Grinberg }
%\address{UMass Boston }
%\email{eric.grinberg@umb.edu }

%\subjclass[2010]{15A03, 97H60, 15A06 } \keywords{Row Reduced Echelon Form, Uniqueness}

\begin{center}
\emph{Dedicated to the memory of my father, Ozias ``Ozi" Grimberg}
\end{center}

\begin{abstract}
Using a left-to-right ``sweeping" algorithm, we define the \emph{Gauche basis} for the column space of a matrix $M$. Interpreting the row reduced echelon form (RREF) of $M$ by \emph{Gauche } means gives a direct proof of its uniqueness. A corollary shows that the (right) null space of $M$ determines its row equivalence class, unmasks a sanitized version of the assertion ``if two  systems are solution equivalent they are row equivalent," and presents the null space as a distinguished graph.  We conclude with pedagogical reflections. 
\end{abstract}

\section{Introduction.}

The row reduced echelon form of a matrix $M$,  RREF$(M)$,  is a useful  tool when working with linear systems \cite{Beez14, lay, Stra14, Stra18}; its uniqueness  is an important property. A survey of papers and textbooks yields a variety of uniqueness proofs. Some are simpler \cite{Yust84} and shorter than others. Generally proofs  begin with two candidates for RREF$(M)$ and conclude that these are equal. It is deemed desirable to have a direct proof, one that simply identifies every atom and molecule of RREF$(M)$ in terms of properties of $M$ and standard conventions. We use the \emph{Gauche basis} of the column space of $M$ to give such a proof,  taking the opportunity to view RREF from a shifted perspective. This context makes it convenient to observe that the (right) null space of $M$ determines its row space, without introducing orthogonality, and yields a near-converse of the familiar assertion ``if two systems are row equivalent then they are solution equivalent." In conclusion we offer some reflections on teaching.

\section{Conventions and  Notations.}
We will work mostly in the vector space $\mathbb F^p$, consisting of $p \times 1$ column vectors with entries in the field $\mathbb F$, and sometimes denote these as transposed row vectors, e.g., 
\(
\begin{pmatrix} 0 \; 1 \cdots 0 \end{pmatrix}^t \). We'll adhere to the ordering conventions of \emph{left to right} and \emph{up to down}. Thus the first column of a matrix is the leftmost, and first entry of a column is its top entry.
 Recall the notation for the ``canonical" or ``standard" basis of $\mathbb F^p$: $\{ \vec e_j \}$ , where $\vec e_j$ stands for the $p \times 1$ column vector
{\small  \( \begin{pmatrix}
 0 & \cdots & 0 & 1 & 0 & \cdots & 0
 \end{pmatrix}^t \) }, 
  with zeros throughout, except for a $1$ in the $j$th entry. 
\smallskip
Recall also that the \emph{span} of a set $S$ of vectors in $\mathbb F^p$ is the collection of all linear combinations of these vectors. Thus the span of the singelton set $\{ \vec v \} $ consists of the set of all scalar multiples of $\vec v$, i.e., a line in $\mathbb F^p$, unless $\vec v= \vec 0$, in which case the span of $\{ \vec v \} $ is $\{ \vec 0 \} $. We also have the degenerate case where $S$ is the empty set; by convention, the span of the empty set is $\{ \vec 0 \} $.

\section{The Remembrance of Row Reduced Echelon Form (RREF).}

Given a matrix $M$, viewed as the coefficient portion of a linear system $M \vec x =\vec b$, we can apply \emph{row operations} to $M$, or to the augmented matrix $( M \vert \vec b )$, and corresponding  equation operations on the system $M \vec x  = \vec b$, to yield a simpler system that is \emph{solution equivalent} to the original. These operations include \emph{scaling a row} by a nonzero scalar, \emph{interchanging} two rows, and \emph{subtracting a scalar multiple of one row from another} row. This last operation is the most commonly used, and is sometimes called a \emph{workhorse} row operation.
\smallskip

Starting with a matrix $M$ and applying carefully chosen row operations, one can obtain  a matrix $E$ with, arguably, the ``best possible" form among all matrices row equivalent to $M$. This is the \emph{row reduced echelon form} of $M$, or RREF$(M)$, or just RREF. We use the definite article \emph{the} because this form turns out to be be unique, as we'll see. 
\smallskip

A matrix $E$ is in RREF if it satisfies the following conditions.

\begin{itemize}
\item \emph{\bf Pivots.}  Sweeping each row of $E$ from the left, the first nonzero scalar encountered, if any, is a~$1$. \\
 We call this entry, along with its column, a \emph{pivot}.
\item \emph{\bf Pivot column insecurity.} 
In a pivot column, the scalar $1$ encountered in the row sweep is the only nonzero entry in its column. 
\item \emph{\bf Downright conventional.}  
If a pivot scalar $1$ is to the right of another, it is also lower down.
\item \emph{\bf Bottom zeros.}
 Rows consisting entirely of zeros, if they appear, are at the bottom of the matrix.
\end{itemize}

Although the RREF conditions may seem labored, a more fluent geometric interpretation will be given below. The term \emph{pivot insecurity} requires explanation. We think of the pivot scalar $1$'s as insecure: they don't want competition from other nonzero entries along their column. Sorry pivots---row insecurity cannot be accommodated.

\section{A Gauche Basis for a Matrix with A Fifth Column.}
For the purpose of introduction and illustration we'll begin with a specific matrix \cite[SAE]{beez20}:
\[
T \equiv \begin{pmatrix}[rrrrr]
2 & 1 & 7 & -7 & 2\\
-3 & 4 &  -5 & -6 &  3\\
1 & 1 & 4 &  -5 & 2
\end{pmatrix}.
\]
We will ``sweep" the columns of $T$ from left to right, and designate each column as a \emph{keeper} or as \emph{subordinate}. These are  meant to be value-neutral, not value judgments, and we hope that no vectors will take offense.  For each column we ask
\begin{gather*}
 \tag{LLQ}\label{eq:LLQ}
 \parbox{25em}{ \begin{center}
\textrm{\emph{ Can we present this column}} 
\newline 
\textrm{\emph{as a linear combination of keeper columns to its left?}}  \end{center} }
\end{gather*}

We will call this the \emph{left-leaning question}, or \emph{LLQ} for short. Columns for which the answer is \emph{no} will be designated as \emph{keepers} and the rest as \emph{subordinates}.

When focusing on the first column of $T$, we recall the convention that a linear combination of the empty set is, in the context of a vector space $V$,  the zero vector of $V$. Thus the $LLQ$ for the first column of $T$ is tantamount to asking: 

\begin{quoting} \centering
\emph{Is this vector nonzero?}
\end{quoting}

\noindent
For $T$ the answer is \emph{yes}. Therefore, we adorn column one with the adjective \emph{keeper}. With the aim of responsible accounting, we ``journal" our action with the vector \( \vec J_1 \equiv \vec e_1 \). (Recall that in our context $\vec e_1$ is the $3 \times 1$ column vector with a $1$ in the first entry and zeros elsewhere.) 

Next, we focus on column two and the $LLQ$, which, in the current context, asks:

\begin{quoting} \centering
\emph{Is this column a scalar multiple of column one?}
\end{quoting}

\noindent
The answer is \emph{no}, so column two is a keeper, and we journal it with $\vec J_2 \equiv \vec e_2$.

 The $LLQ$ for third column asks if this column is a linear combination of the first two (keeper) columns of $T$; by inspection, column three is presentable as a linear combination  of columns one and two, with scalings $3,1$, respectively. So column three is \emph{subordinate} and we journal our action with the vector 
\( \vec J_3 \equiv  3 \vec e_1 + 1 \vec e_2 \),
 which encodes the manifestation of this vector as a linear combination of keeper columns to its left:
 \[
 \begin{pmatrix}[r] 
 7 \\ -5 \\ 4
 \end{pmatrix}
 =
 3 \cdot \begin{pmatrix}[r] 
 2 \\ -3 \\ 1
 \end{pmatrix}
 +
 1 \cdot \begin{pmatrix}[r] 
 1 \\ 4 \\ 1
 \end{pmatrix}
 \, ; \quad
 \vec J_3 \equiv
 \begin{pmatrix}[r] 
 3 \\ 1 \\ 0
 \end{pmatrix}.
 \]

Similarly, the fourth column of $T$ is subordinate, and journaled with 
\[ 
\vec J_4 \equiv {(-2) \cdot \vec e_1 + (-3) \vec e_2 }.
 \]
 The fifth and final column vector of $T$ is not presentable as a linear combination of previous keepers. The reader is invited to prove this or, alternatively, perform a half-turn on the solution box below. \newline

\begin{center}
\rotatebox{180}{
\framebox{
\scalebox{.8}{
\begin{minipage}[c]{0.8\textwidth}
Take $a$ times the first column of $T$ and add it to $b$ times the second column, and look at the top and bottom entries. 
To produce the fifth column of $T$, we need $2a+b=2$ and also $a+b=2$. This implies that $a=0$,
and then we run into trouble with the middle entries of our vectors.
\end{minipage}
}} 
} \end{center}
\smallskip

We declare the fifth column a keeper, at our peril,\footnote{Fifth column--\emph{a group of secret sympathizers or supporters of an enemy that engage in espionage or sabotage within defense lines or national borders}--Merriam-Webster dictionary.} 
and journal it with 
\( \vec J_5 \equiv  \vec e_3 \).
Now form a $3 \times 5$ matrix using the vectors we journaled, in the order we journaled them: 
\(  J \equiv \begin{pmatrix} \vec J_1   & \vec J_2 & \cdots & \vec J_5 \end{pmatrix}  \), or
\begin{equation}
\label{eq:rref_of_5th_column_matrix}
J \equiv 
\begin{pmatrix}[rrrrr]
1 & 0 & 3 &  -2 & 0\\
0 & 1 & 1 & -3 & 0\\
0 & 0 &  0 & 0 & 1
\end{pmatrix}. 
\end{equation}
This turns out to be the $RREF$ of $T$, perhaps surprisingly. For an independent verification, using Gauss--Jordan elimination on the same matrix $T$, see  Example SAE in \cite{beez20}. Notice that our procedure does not show that $(1)$ is row equivalent to $T$, whereas the Gauss--Jordan algorithm, e.g., as in \cite{beez20}, does. It's not difficult  to show directly, in this context, that the Gauche procedure yields a matrix that is row equivalent to the original. In case anyone insists, we will prove this later on; the approach is entirely Gauss--Jordan-esque. 
 
 \section{Beyond the Fifth Column: a General Gauche Algorithm.}

Here we detail the procedure for generating the Gauche basis for an arbitrary matrix and use it to produce the corresponding RREF. For student readers, we suggest following the  ideas  of John H.~Hubbard and Bill Thurston in \emph{How To Read Mathematics}, \cite{Hubb02}:  jump to the illustrative concrete example above, whenever a point in the general procedure below appears 
sinister.
%forbidding.
\smallskip

Let $M$ be a $ p \times q$ matrix (over a fixed field $\mathbb F$). We outline a general algorithm that transforms $M$ into row reduced echelon form without invoking row reduction. This exhibits, among other things, the uniqueness of the row reduced echelon form.
\smallskip
Sweeping the columns of $M$ from left to right, we will adorn some of the columns with the title of \emph{keeper}. Initially, the set of keepers is empty. Going from left to right, we take a column of $M$ and ask the \ref{eq:LLQ}. For the first column of $M$ this is tantamount to asking:
\emph{ Is this column nonzero?}
If so, we declare it a \emph{keeper}  and journal our action with the vector $\vec J_1 \equiv \vec e_1 \in \mathbb F^p $. If the first column is zero, we do not adorn it with the title of keeper; we call it \emph{subordinate} and we journal our action with the vector $\vec J_1 \equiv \vec 0$.  

In general, we examine the $n$th column of $M$ and ask the \ref{eq:LLQ}. If this column is not in the span of the current keeper set, we adorn this column with the \emph{keeper} designation and journal our action with the vector $\vec J_n \equiv \vec e_{\ell +1} $, where $\ell + 1$ is the number of keepers adorned up to this step, current column included. If the current column is presentable as a linear combination of (already designated) keepers, say
$\alpha_1 \vec k_1 + \cdots + \alpha_\ell \vec k_\ell$, where the already designated keeper columns are 
\( \{\vec k_i \}_{i=1}^\ell$, then  we call the current column \emph{subordinate} and journal our action with the vector $\vec J_n \equiv \alpha_1 \vec e_1 + \cdots + \alpha_\ell \vec e_\ell$, recalling that we are focusing on column $n$ and we have $\ell$ keeper columns already designated. The careful  (or fussy) reader may object that the current column may be expressible as a linear combination of keepers in more than one way. However, induction readily shows that at each stage the keeper set is linearly independent.
\smallskip
At the end of this procedure we obtain a matrix $E$ of the same size as $M$.

We will call the algorithm above, transforming $M$ into $E$, the \emph{Gauche procedure} and the resulting basis for the column space of $M$ the \emph{Gauche basis}.

\begin{lemma} The matrix $E$  is in row reduced echelon form.
\end{lemma}
\begin{proof}

In this discussion we will sometimes tacitly identify columns of $E$ with corresponding columns of~$M$.
We take row $i$ of $E$ and ``sweep" it from the left. We encounter a first  nonzero entry in only one circumstance: where we meet a pivot of $E$, i.e.,  a journaled vector $\vec J_\ell $ corresponding to keeper column of $M$. (A nonzero entry in a subordinate column is always assigned only \emph{after} a pivot $1$ entry has already been assigned earlier in the same row.)
 In the Gauche algorithm, whenever we introduce a new journal vector $\vec J_\ell$, corresponding to a pivot, the scalar $1$ appears in a lower slot than those of any prior keepers, and prior subordinate columns are linear combinations of prior keepers, so their entries are zero at this row altitude level as well. 

What about the \emph{downright} condition? If a pivot $1$ is to the right of another, it is also lower down, as it gets adorned with the \emph{keeper} designation at a later stage and is journaled as $\vec e_k$ with a larger value of $k$.

When the pivot journaling stops, no further nonzero entries are journaled in  rows lower than the row of the $1$ entry in the last pivot column. Hence, in particular, all pure-zero rows are at the bottom of $E$.  Thus we have verified that $E$ is in  RREF.
\end{proof}

 \section{Zeroing in on the null space.}
 We now try to redouble our understanding of the meaning of RREF and its relation to the null space of a matrix.
 
 When solving the linear system \mbox{$M \vec x = \vec b$} by row reduction, the matrix $M$ ``calls the shots" and the right-hand side $\vec b$ ``comes along for the ride." That is, the row reduction steps are determined entirely by the coefficient matrix alone, and they are\emph{ applied} to the right-hand side vector. This suggests that RREF is not concerned much with the right-hand side $\vec b$, so we focus on the homogeneous system $M \vec x = \vec 0$, i.e., the \emph{null space}  $\nullspace (M)$.
 
 There is an additional way in which $\nullspace (M)$ figures into our discussion. The Gauche algorithm includes steps that may be called \emph{decisional}: we must decide if a column is a keeper or is subordinate. A close re-reading shows that these decisions may be reinterpreted entirely in terms of the null space of $M$. 
 
 We are led to ponder the question: \emph{Are RREF and Gauche all about $\nullspace (M)$?} Below we will show that this is indeed so: $\nullspace (M)$ determines ``everything." Moreover, we aim to prove this with ``no work at all," somewhat in the spirit of Donald \mbox{J.~Newman's} 1990s \emph{Thought Less (or thoughtless) Mathematics} initiative. Newman sought to systematize a procedure for solving mathematical problems and proving theorems with no ingenuity required at all. The author recalls a colloquium talk delivered by Newman at Temple University, where he gave a \emph{thought less proof} of the infinitude of the primes. A recently published proof, by I.~Mercer (see \cite{Mercer2020}) is reminiscent of Newman's proof. Alas, not much of Newman's thought less initiative is in the literature. But there is this: \cite{Newman93}. Of course, it goes without saying that setting up a thought less proof is not a thought less undertaking.
 
  To further the re-interpretation of RREF plan, we proceed by setting up a small ``dictionary" between linear properties of columns and inclusion properties of the null space. For the benefit of student readers, we point out that small dictionaries are not uncommon in mathematics. (When they get larger, they  turn into \emph{categories} \cite{Maclane97}.)  For instance, in \cite[p. 11]{G-S_90},  we find: \emph{one can set up a ``dictionary" that translates properties of the matrix into optical properties}. \\
 
 After these anticipatory remarks and before implementing proofs we need to add to our notational baggage. 
 In working with columns of $M$ (and of  $E$) we used $\{ \vec e_i \}_{i=1}^p$,  the standard basis of $\mathbb F^p$. Now null$(M)$ is a subset of $\mathbb F^q$ and we'd like to work with the standard basis of this space as well. To avoid (read: reduce) confusion, we'll use the notation $\vec f_i$ for the $q \times 1$ column vector with a $1$ in slot $i$ and zeros elsewhere, so that $\{ \vec f_i \}_{i=1}^q$ is the standard basis of $\mathbb F^q$.

 With this notation we observe that the first column of $M$ is $M \vec f_1$, so asking if the first column of $M$ is zero is tantamount to asking if the vector $\vec f_1$ belongs to the null space of $M$, i.e., if $M \vec f_1 = \vec 0$.  Table \ref{tab:column-null} gives further illustration of this interplay.

  \section{RREF is unique.}
 En route to proving the uniqueness of RREF, we state a lemma which, essentially, asserts that the matrix $E$ comprises the columns of the matrix $M$ written in the Gauche basis of the column space of $M$.
 
 \begin{lemma} \label{heuristic} Let $M$ be a  $p \times q$ matrix over a field $\mathbb F$ and let $E$ be a matrix in RREF which is row equivalent to $M$. 
Let $S \subseteq  \{ 1, \ldots , q \}$ be the index set corresponding to the pivot vectors among the columns of $E$.  Then:
 
 \begin{itemize}
 \item The columns of $M$ corresponding to the index set $S$ form the Gauche basis for the column space of $M$.
 \item Each nonpivot column $\vec c$ of $E$ is a linear  combination of the pivot columns to its left. This combination  exhibits the presentation of the corresponding column of $M$ as a linear combination of Gauche basis vectors to its left. The ``top" entries of $\vec c$   encode this (unique) linear combination, and the rest of the entries of $\vec c$ are ``padded" zeros.
 \end{itemize}
 \end{lemma}
 Note that if $E$ has no pivots at all, then $M=E=0$, which is consistent with the vacuous interpretation of the statement of the lemma. For the matrix $J$ in \eqref{eq:rref_of_5th_column_matrix}, which is an instance of $E$, the set $S$ is $\{ 1,2,5\}$.

 \begin{proof}
 It is well known that if $M$ and $E$ are row equivalent then the associated homogeneous linear systems $M \vec x = \vec 0$ and $E \vec x = \vec 0$ have the same solutions \cite[REMES]{beez20}, \cite[Theorem 3, p. 8]{hoffman-kunze}. That is, $M$ and $E$ have the same (right) null space.  At the risk of slightly abusing language we state a heuristic principle:
\newline

\framebox{ \begin{minipage}{10cm}
\begin{quoting} \centering
 \emph{Every linear property of the columns of $M$ \\ is also enjoyed by the columns of $E$, \\ and conversely.}
 \end{quoting}
 \end{minipage}}
 \newline
 
 This assertion requires some reflection and interpretation. It is inspired, in part, by a deep principle in the analysis of  meromorphic functions \cite{Zalc75}. (See \cite{Radjavi-rosenthal-1997} for a heuristic principle in the context of linear algebra.) Table \ref{tab:column-null} provides illustrations of this heuristic for $M$, and we can do the same for $E$. (Although we captioned the table as a dictionary, we have taken liberties with the language inside; we hope that this is forgivable.)
 
 \begin{table} 
 %
 % Use stackexchange to wrap text within cell within table:https://tex.stackexchange.com/questions/2441/how-to-add-a-forced-line-break-inside-a-table-cell#:~:text=Use%20the%20tabularx%20environment%20instead,line%20breaks%20within%20a%20cell.&text=The%20tabularx%20environment%20has%20a,desired%20width%20of%20the%20table.
 %
 %\vtop{\hbox{\strut top line}\hbox{\strut botline} \hbox{\strut bottomerline}} 
 %
 
\centering
\caption{(column property)$\leftrightarrow$(null space property) Dictionary.} \label{tab:column-null} 
\small
\begin{tabular}{|p{5.8cm}|p{5.9cm}|} 
\hline \textbf{Linear Property of Columns} & \textbf{Inclusion Property of Null Space}\\ & \\ \hline
The first column of $M$ is nonzero.
  & The vector $\vec f_1$ is not in $\textrm{Null}(M)$. \\ 
\hline 
The $k$th column of $M$ is in the span of columns $j_1, \ldots j_\ell$ of $M$. 
& 
\vtop{\hbox{ There exist \, $\alpha_1 , \ldots ,\alpha_\ell$ \, so that}
\hbox{
  $\alpha_1 \vec f_{j_1} + \cdots + \alpha_\ell \vec f_{j_\ell} - \vec f_k \, \in \textrm{Null}(M)$. }}
 \\
\hline 
\vtop{\hbox{Columns $j_1, \ldots , j_\ell$ of $M$} 
\hbox{form a linearly independent set.}} &
\vtop{ \hbox{\strut  For scalars $\alpha_1 , \ldots , \alpha_\ell$ }
\hbox{\strut  the vector $\alpha_1 \vec f_{j_1} + \cdots + \alpha_\ell \vec f_{j_\ell}$ is in $\textrm{Null}(M)$}
\hbox{\strut $\Leftrightarrow$  the scalars $\alpha_1 , \ldots , \alpha_\ell$ all vanish.}} \\
\hline
\end{tabular} 
\end{table}

Iterating the idea, we can express in this null space way the statement 
\begin{quoting} \centering
{ \small
\emph{Columns $j_1, \ldots, j_\ell$ form the  Gauche basis of the column space of $( \cdot )$.} 
}
\end{quoting}
  and others like it. Indeed, in this way, all assertions in the statement of the lemma may be translated into assertions about inclusions in the respective null spaces. Hence these are shared values \cite{Zalc00} for $E$ and $M$.
 \end{proof}

 \begin{theorem} Let $M$ be a matrix. Then there is one and only one matrix $E$ in RREF that is row equivalent to $M$.
 \end{theorem} 
 \begin{proof}
The lemma above describes every entry of $E$ in terms of left-down conventions and properties of $M$, without reference to any process for row reducing $M$ to yield $E$, e.g., Gauss--Jordan elimination. This proves uniqueness.  For existence, one can invoke the Gauss--Jordan algorithm, or prove directly (and, admittedly, with Gauss--Jordan-esque ideas) that $E$ is row equivalent to $M$, as is done independently, below in Proposition~\ref{gauche_is_row_equivalent}. \end{proof}

\begin{corollary}
The null space of a matrix $M$ determines the RREF and the row space of $M$. Hence if two matrices  of the same size have the same null space, then they are row equivalent.
\end{corollary}

\begin{proof}

The matrix $M$ has a unique RREF and its Gauche construction uses only the null space of $M$.
\end{proof}
The relation between the null space and the row space of a matrix is well known and does not require the concept of orthogonality.  This is mentioned repeatedly in \cite{hoffman-kunze}, at times concretely in examples, at times in generality, but in passing. It is worthy of further promulgation.

\section{Geometric User Inferface (GUI)}
We take heed of \cite{shifrin-adams} and affirm that, while linear algebra is algebraic, it is geometric as well. Thus the uniqueness of RREF, expressed algebraically above, may be viewed geometrically as well.  The Gauche path to the RREF of a matrix  e.g., $T$, presents the null space of the original matrix ($T$) as a \emph{graph} over the vector subspace spanned by ``axes" corresponding to the subordinate, or nonpivot columns of $T$. We can read \eqref{eq:rref_of_5th_column_matrix} to say that the null space of $T$ is the graph over the span of the third and  fourth axes of $\mathbb R^5$ given by the relations \\

\begin{minipage}{0.25\textwidth} 
{\large \[  \systeme*{
  x_1 =  -3x_3+2x_4 ,
  x_2 = - \,\,\, x_3 \, +3x_4 ,
  x_5 = \phantom{-} 0x_3+0x_4 }
  \, .  \] }
\end{minipage} 
\quad\quad\quad\quad\quad\quad\qquad\qquad
{\begin{minipage}[c]{0.45\textwidth}
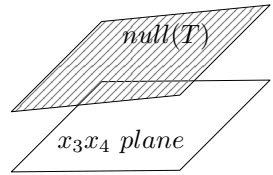
\begin{tikzpicture}[x  = {(-0.5cm,-0.5cm)},
                    y  = {(0.9659cm,-0.25882cm)},
                    z  = {(0cm,1cm)},
                    scale = 0.8,
                    color = {black}]

\draw[
draw=black,very thin,line join=round]
 (0,0,1) --
  (1,0,1) --  node  { $\qquad \qquad \qquad  { x_3x_4 \, \,plane} $}
 (3,0,1) -- 
 (3,{3*cos(15)},{3*sin(15)+1}) --
 (0,{3*cos(15)},{3*sin(15)+1}) --cycle ;
 \draw[fill=lightgray,pattern=north east lines, pattern color=gray, 
        draw=black,very thin,line join=round]
 (0,0,2) -- node { $\qquad \qquad \qquad  null (T)$}
 (1,0,2) --
 (3,0,2) --
 (3,{3*cos(15)},{3*sin(15)+2.25}) --
 (0,{3*cos(15)},{3*sin(15)+2.25}) --cycle  ;

\end{tikzpicture} \end{minipage}} \\

Among all the different ways to present the null space of $T$ as a graph (within the Euclidean space with axes corresponding to columns of $T$), the RREF way employs as a base  the span of the ``rightmost" axes available for the task.  Why \emph{rightmost}, the reader may ask, given the Gauche perspective? The RREF exercises \emph{leftmost} selection of pivot columns, making nonpivot columns rightmost. The nonpivot columns of $T$ correspond to free variables for solutions of the linear system $T \vec x = \vec 0$ and the pivot columns correspond to dependent variables. This is tantamount to presenting $\text{null} (T)$ as a graph.

In the article \cite{lay}, D.~C.~Lay points out that vector subspaces of Euclidean space are usually presented as either the locus of solutions to a homogeneous system of linear  equations or the span of a collection of vectors, and offers algorithms to link the two presentation types. All the algorithms involve RREF and may be viewed as presenting the vector subspace as a graph over the rightmost span of axes available.

\section{The solution determines the problem.}

In the television game show \emph{Jeopardy!} contestants are given answers and asked to guess the questions from whence they came. In calculus we introduce anti-derivatives as ``differentiation Jeopardy." The following linear-algebraic Jeopardy variant may be considered:

\smallskip

\framebox{ \begin{minipage}{10cm}
\begin{quoting} \centering %\small
\emph{If two linear systems have the same solution set, then \\ they are row equivalent.}
\end{quoting}
\end{minipage}}

\smallskip

\smallskip
Literally, as stated, this assertion is manifestly false. (Please do not  invoke it out of context.)  For suppose we have two inconsistent linear systems. They both have the empty set of solutions, hence the same set of solutions. But the two systems may not have the same number of equations. They may even involve different variables. Clearly, we need to focus on consistent linear systems of the same size. We will also tacitly assume that they involve the same unknowns.

\begin{corollary}
If two consistent linear systems of the same size are solution equivalent, then they are row equivalent.
\end{corollary}

\begin{proof}
First assume that the systems are homogeneous. Then the hypothesis says that the corresponding matrices have the same null space. Hence, by the previous corollary, they have the same RREF and are thereby row equivalent. In the general case, simply note that the solution set of a (possibly) inhomogeneous linear system consists of one particular solution added to the solution space of the associated homogeneous system.
\end{proof}

\section{An Existential Question.}

The Gauche procedure takes a matrix $M$ and associates with it a matrix $E$ that is in RREF. But how do we know that there exists a sequence of row operations taking $M$ to $E$, i.e., why is $E$ row equivalent to $M$?  
\smallskip

We can invoke the Gauss--Jordan elimination algorithm which yields a matrix in RREF that is row equivalent to $M$ and then cite uniqueness considerations to conclude that our Gauche $E$ must be that matrix. But this is unsatisfying---we should be able to show directly  that the Gauche-produced matrix $E$ is row-equivalent to $M$ and, if one insists, we can.

\begin{prop} \label{gauche_is_row_equivalent}
For a matrix $M$, the Gauche-produced RREF matrix $E \equiv E(M)$ is row equivalent to $M$.
\end{prop}
\begin{proof}

We can take $M$ and row reduce it to yield the Gauche-produced matrix $E$ following the  algorithm illustrated below:
\[ 
\tiny
\begin{split}
%\scriptsize
\begin{pmatrix}[rrrrrr]
0        & \cdots  0       & *       & * & \cdots & * \\
0        & \cdots  0       & *       & * & \cdots & * \\
 \vdots & \vdots  \vdots & \vdots & * & \cdots & * \\
0        & \cdots  0       & \ne 0  & * & \cdots & * \\
 \vdots & \vdots  \vdots & \vdots & * & \cdots & * \\
0        & \cdots  0       & *       & * & \ldots & * \\
\end{pmatrix}
\longrightarrow
\begin{pmatrix}[rrrrr]
0        & \ldots  0       & \ne 0  & \ldots & * \\
0        & \ldots  0       & *       & \ldots & * \\
 \vdots & \vdots  \vdots & \vdots & \ldots & * \\
0        & \ldots  0       &  *      & \ldots & * \\
 \vdots & \vdots  \vdots & \vdots & \ldots & * \\
0        & \ldots  0       & *        & \ldots & * \\
\end{pmatrix}
%\longrightarrow
\\
\longrightarrow 
\begin{pmatrix}[rrrrr]
0        & \ldots  0       & 1       & \ldots & * \\
0        & \ldots  0       & *       & \ldots & * \\
 \vdots & \vdots  \vdots & \vdots & \ldots & * \\
0        & \ldots  0       &  *      & \ldots & * \\
 \vdots & \vdots  \vdots & \vdots & \ldots & * \\
0        & \ldots  0       & *        & \ldots & * \\
\end{pmatrix}
\longrightarrow
\begin{pmatrix}[rrrrr]
0        & \ldots  0       & 1       & \ldots & * \\
0        & \ldots  0       & 0       & \ldots & * \\
 \vdots & \vdots  \vdots & \vdots & \ldots & * \\
0        & \ldots  0       &  0      & \ldots & * \\
 \vdots & \vdots  \vdots & \vdots & \ldots & * \\
0        & \ldots  0       & 0       & \ldots & * \\
\end{pmatrix}
\longrightarrow \cdots
.   \end{split}
\]

If $M$ is the zero matrix, then $E=M$ and we are done. Otherwise, $E$ has a first pivot column, which corresponds to the first nonzero column of $M$, say column $j_1$. Taking $M$ and permuting rows, we obtain a matrix whose first nonzero column is number $j_1$, and which has a nonzero entry in the first slot; after scaling the first row we can assume that this entry is $1$. Subtracting scalar multiples of the first row from each of the other rows, i.e., employing workhorse row operations, we obtain a matrix whose first nonzero column is the $j_1$st, with entries equal to those of $\vec e_1$. If $E$ has no other pivot columns, then all later columns are scalar multiples of the $j_1$st, and we are done. If $E$ has a second pivot column, say in slot $j_2$, then this column must have a nonzero entry below the first pivot. Permuting rows other than the first and then applying workhorse-type operations and rescaling the top nonzero entry in this column, we obtain $\vec e_2$ in the $j_2$nd slot while retaining $\vec e_1$ in the first slot. Continuing this way, we produce row operations that place appropriate canonical vectors of the form $\vec e_\ell$ in each of the pivot slots. Each of the nonpivot columns is a linear combination of the pivot columns to its left, and requires no additional ``processing" by row operations. Thus we have exhibited $E\equiv E(M)$ as the result of a sequence of row operations applied to $M$.
\end{proof}

\section{Reflections on Teaching.}

The method of elimination via row reduction may be introduced at the very start of a course on linear algebra. Taking the Gauche approach to echelon form, we are led naturally, directly, and concretely to the notions of \emph{linear combination}, \emph{span}, and \emph{linear independence}. Definition and application are threaded---no need for a separate introduction with rationale for use. This brings to mind a parallel in a \emph{Math Proof} course. Every such course covers Euclid's proof of the infinitude of primes, and rightly so. But we can also add H.~Furstenberg's ``topological" proof \cite{Aign18, Furst55}.   Fursternberg's proof leads directly to the basic set operations of intersection, union and complement. Here too, definition and application are threaded and allied; motivation is built in. True, a direct reading of Furstenberg's proof does require some familiarity with topology, possibly turning the motivation upside down. And there is a variant of Furstenberg's proof that does not require topological notions: \cite{Mercer2009}. Then again, the topological aspect of the proof may be regarded as a teaching feature, not a bug, anticipating notions to come in later courses. Also, this proof requires no theorems in \emph{topology}, but only the definition of the term. The challenge, then, is to introduce the concept of open set in a brief, self contained, pedagogically sound manner, so as to pave the way for Fursternberg's proof early in a proofs course. Here we have tried to address the linear algebraic analogy, which is easier.
\smallskip

We conclude with a question: Is there a \emph{book proof} (see \cite{Aign18}) of the uniqueness of RREF? Is the fact worthy of inclusion in \emph{The Book}?

%\medskip 
\section*{Acknowledgments}
The Gauche idea emerged from a conversation with Professor Gilbert Strang in the fall of 2019 at MIT's \emph{Endicott House}. The author is grateful to Professor Strang for the conversation and for his inspiring writings through the years. He is also grateful to MIT for the invitation to the Endicott House event, and he would happily repeat the experience. In addition, he thanks all who have commented on the paper; their suggestions added value and are much appreciated.


\begin{thebibliography}{1978}

\bibitem{Aign18}  Aigner, M., Ziegler, G. (2018). \emph{Proofs from The Book},  6th ed. (Hofmann, K. H., illust.) Berlin: Springer-Verlag. 

\bibitem{Beez14} Beezer, R. A.  (2014). Extended echelon form and four subspaces. 
\emph{Amer. Math. Monthly}. 121(7): 644--647.

\bibitem{beez20} Beezer, R. A. (2016). \emph{A First Course in Linear Algebra}, edition 3.50, online open-source edition 3.50. \url{linear.ups.edu}
% accessed April 15, 2020.

%\bibitem{fifth_column} Dictionary defiinition of \emph{fifth column}: a group within a country at war who are sympathetic to or working for its enemies. Merriam-Webster.

\bibitem{Furst55} Furstenberg, H. (1955).
On the infinitude of primes. \emph{Amer. Math. Monthly}. 62(5): 353.

\bibitem{G-S_90}  Guillemin, V., Sternberg, S. (1990). \emph{Symplectic Techniques in Physics}. Cambridge, UK: Cambridge Univ. Press.

\bibitem{hoffman-kunze}
Hoffman, K., Kunze, R. (1971). \emph{Linear Algebra}, 2nd ed.  Englewood Cliffs, NJ:
Prentice-Hall. 

 
 \bibitem{Hubb02} Hubbard, J. (2002). Reading mathematics. In:
 Hubbard, J. H., Burke Hubbard, B., \emph{Vector Calculus, Linear Algebra, and Differential Forms: A Unified Approach}, 2nd ed. Englewood Cliffs, NJ: Prentice Hall;
chapter 0.1:  \url{ pi.math.cornell.edu/~hubbard/readingmath.pdf}
% Accessed April, 2020.

\bibitem{lay} Lay, D. C. (1993). Subspaces and echelon forms. \emph{Coll. Math. J.} 24(1): 57--62. {\url doi.org/10.1080/07468342.1993.11973507}

\bibitem{Maclane97} Mac Lane, S. (1997).  Review of \textit{Conceptual Mathematics: A First Introduction} by F. William Lawvere and Steven Schanuel, \emph{Amer. Math. Monthly}.
104(10): 985--987. {\url doi.org/10.1080/00029890.1997.11990751}

\bibitem{Mercer2009} Mercer, I. (2009). On Furstenberg's proof of the infinitude of primes. \emph{Amer. Math. Monthly}. 116(4): 355--356.  
{\url doi.org/10.1080/00029890.2009.11920947}

\bibitem{Mercer2020} Mercer, I. (2020). Another proof of the infinitude of primes.  \emph{Amer. Math. Monthly}. 127(10): 938. 
{\url doi.org/10.1080/00029890.2020.1815482}

\bibitem{Newman93} Newman, D.J. (1993). Thought less mathematics. In: Gale, D., Newman, D. J. Mathematical entertainments. 
\emph{Math. Intelligencer}. 15: 58--61.  {\url www.cut-the-knot.org/blue/OddballProblem2.shtml }

\bibitem{Zalc00} 
 Pang, X., Zalcman, L. (2000). Normal families and shared values. \emph{Bull. London Math. Soc.} 32(3): 325--331. 

\bibitem{Radjavi-rosenthal-1997} Radjavi, H., Rosenthal, P. (1997). From local to global triangularization. \emph{J. Funct. Anal.} 147: 443--456.

\bibitem{shifrin-adams}
 Shifrin, T., Adams, M. R. (2011). \emph{Linear Algebra, A Geometric Approach}, 2nd ed.  New York, NY:W.~H.~ Freeman and company.

\bibitem{Stra14} Strang, G. (2014). The core ideas in our teaching.  \emph{Notices Amer. Math. Soc.} 61(10): 1243--1245.  

 \bibitem{Stra18} Strang, G.  (2018).  Multiplying and factoring matrices. \emph{Amer. Math. Monthly}. 125(3): 223--230. 
 
\bibitem{Yust84}  Yuster, T.   (1984).  The reduced row echelon form of a matrix is unique: a simple proof. \emph{ Math. Mag.} 57(2): 93--94.

\bibitem{Zalc75} Zalcman, L. (1975). A heuristic principle in complex function theory.
\emph{ Amer. Math. Monthly}. 82(8): 813--817.

 \end{thebibliography}
 \end{document}